\documentclass[reqno,12pt]{amsart}

\usepackage{amssymb, amsmath,graphicx}
\usepackage{verbatim,cite}
\usepackage{hyperref,xcolor}

\newcommand{\R}{\mathbb{R}}
\renewcommand{\P}{\mathbb{P}}
\newcommand{\G}{\mathbb{G}}

\renewcommand{\iff}{\Leftrightarrow}

\theoremstyle{plain}
\newtheorem{theorem}{Theorem}
\newtheorem{lemma}[theorem]{Lemma}

\newtheorem{proposition}[theorem]{Proposition}

\theoremstyle{definition}
\newtheorem{definition}[theorem]{Definition}

\newtheorem{example}[theorem]{Example}

\theoremstyle{remark}

\begin{document}
	\title{Minimal Embedding Dimensions of Connected Neural Codes}
	\author{Raffaella Mulas}
	\address{Max Planck Institute for Mathematics in the Sciences, D-04103 Leipzig, Germany}
	\email{raffaella.mulas@mis.mpg.de}
	\author{Ngoc M Tran}
	\address{Department of Mathematics, The University of Texas at Austin, TX 78751, USA and the Hausdorff Center for Mathematics, Bonn University, D-53115 Bonn, Germany }
	\email{ntran@math.utexas.edu}
	
	\maketitle
	\begin{abstract}
    Receptive field code is a recently proposed deterministic model of neural firing. The main question is to characterize the set of realizable codes, and their minimal embedding dimensions with respect to a given family of receptive fields. Here we answer both of these questions when the receptive fields are connected. In particular, we show that all connected codes are realizable in dimension at most three. To our knowledge, this is the first family of receptive field codes for which both the exact characterization and minimal embedding dimension are known. 
	\end{abstract}
	
	\section{Introduction}
	The \textit{receptive field code} is a deterministic model of neural firing defined by Curto, Itskov, Veliz-Cuba and Youngs \cite{ring}. It consists of $n \in \mathbb{N}$ neurons, each neuron $i \in [n] = \{1, 2, \ldots, n\}$ has a receptive field $U_i \subseteq \R^d$. Given a stimulus $x \in \R^d$, the neurons generate a codeword $\sigma(x) \subseteq 2^{[n]}$ via
	\begin{equation}\label{eqn:code.deterministic}
	i \in \sigma(x) \iff x \in U_i.     
	\end{equation}
	A receptive field code $\mathcal{C}(\mathcal{U}) \subseteq 2^{[n]}$ is the set of all possible codewords generated from the collection of receptive fields $\mathcal{U} = (U_1, \ldots, U_n)$. For convenience, we will assume that every receptive field code includes the empty set, i.e. $\emptyset\in\mathcal{C}$, which is equivalent to assume that $\bigcup_{i\in[n]}U_i\subsetneq\mathbb{R}^d$. The results in this paper still hold if one assumes $\emptyset\notin\mathcal{C}$.
	\begin{figure}[h]
		\begin{center}
			\includegraphics[width=7cm]{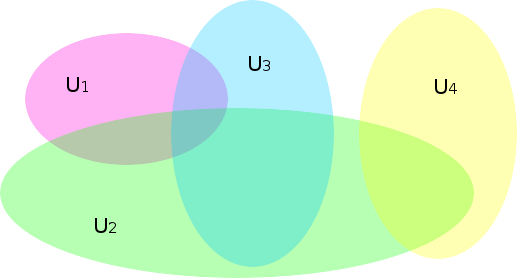}
		\end{center}
		\caption{$4$ receptive fields generating the receptive field code $\mathcal{C}(\mathcal{U})=\{\emptyset,1,2,3,4,12,13,23,24,123\}$.}
	\end{figure}\newline
	Given a code $\mathcal{C} \subseteq 2^{[n]}$ with $\emptyset\in\mathcal{C}$ and a family $\mathcal{F}_d$ of sets in $\R^d$, we say that $\mathcal{C}$ is realizable in dimension $d$ if $\mathcal{C} = \mathcal{C}(\mathcal{U})$ for some $\mathcal{U} \subseteq \mathcal{F}_d$. Call the smallest such $d$ the minimal embedding dimension of the code $\mathcal{C}$, denoted $d^\ast(\mathcal{C},\mathcal{F})$. The \emph{minimal embedding problem} is to find $d^\ast(\mathcal{C},\mathcal{F})$ for given a code $\mathcal{C} \subseteq 2^{[n]}$ and a family $\mathcal{F} = (\mathcal{F}_d, d \geq 1)$ of sets in $\R^d$. 
	
	This paper focuses on \emph{connected codes}. These are codes realizable by connected sets in $\R^d$, for some $d \in \mathbb{N}$, which are either all closed or all open. Asking for these sets to be either all closed or all open makes sense in neuroscience, where receptive fields are intrinsically noisy \cite{Giusti-Itskov}, and it also makes the problem non-trivial since all codes can be realized with connected sets if there are no other constraints. Furthermore, a code is realizable with closed connected sets if and only if it is realizable with open connected sets. In fact, if $\mathcal{C} = \mathcal{C}(\mathcal{U})$ where $\mathcal{U}=(U_1, \ldots, U_n)$ is given by closed connected sets, then $\mathcal{C} = \mathcal{C}(\mathcal{U'})$ where $\mathcal{U'}:=(U'_1, \ldots, U'_n)$ and $U'_i$ is a sufficiently small open connected set containing $U_i$. Vice versa, if $\mathcal{C} = \mathcal{C}(\mathcal{U})$ where $\mathcal{U}=(U_1, \ldots, U_n)$ is given by open connected sets, then $\mathcal{C} = \mathcal{C}(\mathcal{U''})$ where $\mathcal{U''}:=(U''_1, \ldots, U''_n)$ and $U''_i$ is a sufficiently big closed connected set contained in $U_i$.
	
	Our main results completely characterize realizability and minimal embedding dimensions of connected codes. In particular, it is easy to check if a code is connected, and if it is, then the minimal embedding dimension is at most $3$. The graph of a family of connected sets $\mathcal{U}$ is given in Definition \ref{defn:graph.of.u}. 
	
	An alternative characterization of connected codes can be found in \cite[Theorem 4.1]{Jeffs}.
	\begin{proposition}[Realizability of connected codes]\label{prop:main}
		A code $\mathcal{C}$ is connected if and only if for each $\sigma,\tau\in\mathcal{C}$ and for each $i\in\sigma\cap\tau$, there exists a sequence of distinct codewords $\nu_1,\dots,\nu_m\in\mathcal{C}$ such that:
		\begin{itemize}
			\item $\sigma=\nu_1$
			\item either $\nu_j\subseteq\nu_{j+1}$ or $\nu_{j+1}\subseteq\nu_{j}$, for every $j\in[m-1]$
			\item $\nu_m=\tau$
			\item $i\in\nu_j$ for each $j\in[m]$.
		\end{itemize}
	\end{proposition}
	\begin{theorem}[Minimal embedding of connected codes]\label{thm:main}
		Suppose $\mathcal{C}$ is a connected code. Let $d^\ast(\mathcal{C})$ denote its minimal embedding dimension with respect to the family of connected sets. 
		\begin{itemize}      
			\item $d^\ast(\mathcal{C}) = 1$ if and only if the sensor graph of $\mathcal{C}$ is bipartite \cite{rosen-zhang}. 
			\item Else, $d^\ast(\mathcal{C}) = 2$ if and only if there exists a realization $\mathcal{C}(\mathcal{U}) = \mathcal{C}$ by connected sets $\mathcal{U}$ in dimension $3$ such that the graph of $\mathcal{U}$ is planar.
			\item Else, $d^\ast(\mathcal{C}) = 3$. 
		\end{itemize}
	\end{theorem}
	
	In dimension 1, connected codes equal the convex codes studied by Rosen and Zhang \cite{rosen-zhang}. The characterization for $d^\ast =1$ via the sensor graph in Theorem~\ref{thm:main} belongs to \cite{rosen-zhang}, and is included for completeness. We do not define the sensor graph of a code here, but note that it is an intrinsic property of the code, independent of any realization $\mathcal{U}$. 
	
	The minimal embedding dimension $d^\ast(\cdot, \mathcal{F})$ and the family $\mathcal{F}$ of receptive fields form a trade-off in measuring the complexity of the signal encoded by the neurons, and is thus of particular interest in receptive field coding. There has been a number of work on criterion for realizability and bounds for $d^\ast(\cdot,\mathcal{F})$ when the set $\mathcal{F}$ consists of (open or closed) convex sets \cite{Curto-Gross,localobstructions,Giusti-Itskov,rosen-zhang,Lienkaemper}. 
	However, complete characterization and the exact minimal embedding dimension of convex codes remain a problem. Giusti and Itskov \cite{localobstructions} found necessary conditions for a code to be realizable with open convex sets, and proved lower bounds on the embedding dimensions of such codes. In \cite{Giusti-Itskov}, Cruz, Giusti, Itskov and Kronholm proved that there exists a family of codes, called max-intersection-complete codes, that are both open convex and closed convex, and they gave an upper bound for their embedding dimension. To the best of our knowledge, connected codes form the first family of receptive field codes for which an intrinsic characterization and the exact embedding dimension is known. Furthermore, our proof gives explicit constructions for the code realization in each dimension.
	
	Receptive field codes are closely related to Euler diagrams, which found applications in information systems, statistics and logic \cite{2DEuler,Rodgers}. Since their main applications are in visualization, the literature on Euler diagrams focus exclusively on $2$ and $3$ dimensions. Translated to our setup, an Euler diagram in $\R^3$ is a collection $\mathcal{U} = (U_1, \ldots, U_n)$ of closed, orientable surfaces embedded in $\R^3$. An Euler diagram in $\R^2$ is a similar collection of closed curves embedded in $\R^2$. A diagram description is a code $\mathcal{C}$ such that $\emptyset \in \mathcal{C}$. The description of an Euler diagram $\mathcal{U}$ is the code $\mathcal{C}(\mathcal{U}^\circ)$, where $\mathcal{U}^\circ:=(U^\circ_1, \ldots, U^\circ_n)$ consists of the relative interior of the sets $U_i$'s. The main problem in this literature is realizability: given a code $\mathcal{C}$, is there an Euler diagram $\mathcal{U}$ such that $\mathcal{C}=\mathcal{C}(\mathcal{U}^\circ)$? 
	Every code $\mathcal{C}$ can be realized by an Euler diagram in dimension $2$ \cite{Rodgers}, and by an Euler diagram in dimension $3$ with connected sets $U_i$'s \cite{Drawability}. Note the crucial difference to receptive field codes: in an Euler diagram, codewords are generated by intersection of the relative interior of the $U_i$'s. In particular, all codes $\mathcal{C}$ which fail the condition of Proposition \ref{prop:main} satisfy $\mathcal{C} = \mathcal{C}(\mathcal{U}^\circ)$ for some tuple of closed connected sets $\mathcal{U}$ in $\R^3$, but $\mathcal{C} \neq \mathcal{C}(\mathcal{U})$ for any tuple of closed connected sets in any dimension.
	
	In practice, neural firing is stochastic. One could incorporate noise to the receptive field code by replacing the deterministic equation (\ref{eqn:code.deterministic}) with some parametrization of the firing probability $\P(i \in \sigma(x)|x \in U_i)$. To be well-defined, this model needs further specifications, such as the distribution of the signal on $\R^d$. In this formulation, the minimal embedding dimension is a difficult and poorly formed statistical problem. Furthermore, it is clear that the minimal embedding dimension depends heavily on such details. However, underlying such models the assumption that there is a set of true receptive fields $\mathcal{U}$. Knowing the minimal embedding dimension for the deterministic model ensures that the neuroscientist do not have excessively many parameters, which can lead to ill-defined estimation problems. From this view, Theorem \ref{thm:main} states that if the true receptive fields are only required to be connected, one can assume that they are in dimension $3$.
	
	Apart from connected and convex sets, there are many biologically relevant models for receptive fields. Finding the minimal embedding dimension of receptive field codes realizable by any given family is an interesting and challenging problem. To be concrete, we propose another simple family motivated by observations from neuroscience. In experiments, one often encounter a group of neurons which all have the same receptive field up to translation, such as the retinal ganglion cells, head direction cells \cite{Dayan-Abbott}, place cells and grid cells  \cite{OKeefe1,Moser}. This corresponds to the case where $\mathcal{F}_d$ consists of all possible translations of some set $S \subset \R^d$. We call this the \emph{shift} code. Thus, a concrete open problem is: which shift codes can be realized, and what would be their minimal embedding dimensions? 
	
	\section{Proof of the main results}\label{sec:main}
	
	\begin{definition}
		Let $\mathcal{C}$ be a code on $n$ neurons. We say that $\mathcal{C}$ is realizable by an atom sequence $\mathcal{A} = (A_\sigma \subseteq \R^d, \sigma \subseteq [n])$ if $A_\sigma \neq \emptyset \iff \sigma \in \mathcal{C}$. In this case, write $\mathcal{C} = \mathcal{C}(\mathcal{A})$. 
	\end{definition}
	
	\begin{lemma}\label{lem:u.from.a}
		Let $\mathcal{C}$ be a code on $n$ neurons. Then $\mathcal{C} = \mathcal{C}(\mathcal{A})$ if and only if $\mathcal{C} = \mathcal{C}(\mathcal{U})$, where
		\begin{equation}\label{eqn:u.from.a}
		U_i = \bigcup_{i\in\sigma} A_\sigma, 
		\end{equation} 
		or equivalently,
		\begin{equation}\label{eqn:a.from.u}
		A_\sigma=\Biggl(\bigcap_{i\in\sigma}U_i\Biggr)\backslash\bigcup_{j\notin\sigma}U_j,
		\end{equation}with the convention that $A_\emptyset=\mathbb{R}^d\backslash\bigcup_{i\in[n]}U_i$.\newline
		In other words, $\mathcal{A}$ and $\mathcal{U}$ determine each other. 
	\end{lemma}
	
	\begin{lemma}\label{lem:always.realizable}
		Let $\mathcal{C}$ be a code on $n$ neurons. For any $d \geq 1$, $\mathcal{C} = \mathcal{C}(\mathcal{U})$ for some sequence of sets $\mathcal{U}$ in $\R^d$.
	\end{lemma}
	\begin{proof}
		It is sufficient to prove this statement for $d = 1$. For each $\sigma \in \mathcal{C}\backslash\{\emptyset\}$, let $A_\sigma = \bigcup_{i \in \sigma}\{i\} \subset \R$, and let $A_\emptyset=\mathbb{R}\backslash\bigcup_{\sigma \in \mathcal{C}\backslash\{\emptyset\}}A_\sigma$. Then $\mathcal{C}$ is realized by the atom sequence $\mathcal{A}$. Define $\mathcal{U}$ via (\ref{eqn:u.from.a}). By Lemma \ref{lem:u.from.a}, $\mathcal{C} = \mathcal{C}(\mathcal{U})$. 
	\end{proof}
	
	\begin{definition}
		We say that two sets $A,B \subset \R^d$ are adjacent if $A \cap B = \emptyset$ and either $\overline{A} \cap B \neq \emptyset$ or $A \cap \overline{B} \neq \emptyset$, where $\overline{A}$ denotes the closure of $A$ in the Euclidean topology.
	\end{definition}
	\begin{definition}[The graph of a realization]\label{defn:graph.of.u}
		Let $\mathcal{C} = \mathcal{C}(\mathcal{U})$ be a connected code with realization $\mathcal{U}$.  Let $\mathcal{A}$ be the atoms defined via $\mathcal{U}$ in (\ref{eqn:a.from.u}). The graph of $\mathcal{U}$, denoted $\G(\mathcal{U})$, is a graph with one vertex for every connected component of each atom $A_\sigma$ with $\sigma\neq\emptyset$, and an edge for every pair of connected components of atoms that are adjacent.
	\end{definition}
	\begin{lemma}\label{lem:g.in.dprime}
		Let $\mathcal{C} = \mathcal{C}(\mathcal{U})$ be a connected code with realization $\mathcal{U}$ in dimension $d$. If $\G(\mathcal{U})$ can be embedded in $\R^{d'}$, then $\mathcal{C}$ can also be realized by connected sets in dimension $d'$. 
	\end{lemma}
	\begin{proof}
		Take an embedding of $\G(\mathcal{U})$ in $\R^{d'}$. Let $\mathcal{A}$ be the atoms defined via $\mathcal{U}$ in (\ref{eqn:a.from.u}). Let $v_{\sigma j} \in \R^{d'}$ be the realization of the vertex of $\G(\mathcal{U})$ indexed by the $j$-th component of the atom $A_\sigma$. For each pair of nodes $v_{\sigma j}$ and $v_{\tau k}$, let $e_{\sigma j, \tau k} \subset \R^{d'}$ be the realization of the edge between these nodes. 
		If they are not connected, define $e_{\sigma j, \tau k} = \emptyset$. Now define atoms $\mathcal{A}'$ in $\R^{d'}$ via 
		\begin{equation*}
		A'_\sigma:=\bigcup_j\Biggl(v_{\sigma j}\cup\bigcup_{\tau k: \sigma\subset\tau}e_{\sigma j,\tau k}\Biggr) \subset \R^{d'},
		\end{equation*}for $\sigma\in\mathcal{C}\backslash\{\emptyset\}$, and
		\begin{equation*}
		A'_\emptyset:=\R^{d'}\backslash\bigcup_{\sigma\in\mathcal{C}\backslash \{\emptyset\}}A'_\sigma.
		\end{equation*}
		It is easy to check that $\mathcal{C} = \mathcal{C}(\mathcal{A}')$, so $\mathcal{C}$ is realizable in dimension $d'$, as needed.  
	\end{proof}
	
	\subsection{Proof of Proposition \ref{prop:main}}
	Let $\mathcal{C}$ be a code on $n$ neurons. By Lemma~\ref{lem:always.realizable}, $\mathcal{C} = \mathcal{C}(\mathcal{U}) = \mathcal{C}(\mathcal{A})$ for some $U_i, A_\sigma \subseteq \R^d, i \in [n], \sigma \subseteq [n]$. For each $i \in [n]$, $U_i$ is connected if and only if for every $\sigma,\tau \subseteq [n]$ such that $i\in\sigma\cap\tau$, from each connected component $C_\sigma$ of $A_\sigma$ to each connected component $C_\tau$ of $A_\tau$ there is a path 
	$C_\sigma = C_{\nu_1} \to  C_{\nu_2} \ldots \to  C_{\nu_{m-1}} \to C_{\nu_m} = C_\tau$ in $\G(\mathcal{C}(\mathcal{U}))$, where $C_{\nu_j} \subseteq A_{\nu_j}$ is a connected component of $A_{\nu_j}$, such that $A_{\nu_j}\subseteq U_i$ for each $j\in[m]$. Note that, in order to have the receptive fields either all open or all closed, two connected components $C_\sigma\subseteq A_\sigma$, $C_\tau\subseteq A_\tau$ are allowed to be adjacent if and only if either $\sigma\subseteq\tau$ or $\tau\subseteq\sigma$.
	Hence $U_i$ is allowed to be connected if and only if for every $\sigma,\tau$ such that $i\in\sigma\cap\tau$, there exists a sequence of distinct codewords $\nu_1,\dots,\nu_m\in\mathcal{C}$ such that:
	\begin{itemize}
		\item $\sigma=\nu_1$
		\item either $\nu_j\subseteq\nu_{j+1}$ or $\nu_{j+1}\subseteq\nu_{j}$, for every $j\in[m-1]$
		\item $\nu_m=\tau$
		\item $i\in\nu_j$ for each $j\in[m]$.
	\end{itemize}This proves the proposition. 
	\qed
	
	\subsection{Proof of Theorem \ref{thm:main}}
	We split the statement of Theorem \ref{thm:main} into two parts, and prove them separately. The first part, Proposition \ref{prop:d3} states that the minimal embedding dimension for a connected code is at most $3$. The second part, Proposition \ref{prop:d2} gives a characterization for connected codes with $d^\ast = 2$. For the case $d^\ast=1$, see \cite[Proposition 1.9 and Theorem 3.4]{rosen-zhang}.
	\begin{proposition}\label{prop:d3}
		Let $\mathcal{C}$ be a connected code on $n$ neurons. Then $\mathcal{C}$ is realizable by connected sets in dimension $3$.
	\end{proposition}
	\begin{proof}
		For all $\sigma\in\mathcal{C}\backslash\{\emptyset\}$, choose disjoint balls $B_\sigma\subset\mathbb{R}^3$ and for all $\sigma,\tau\in\mathcal{C}$ such that $\sigma\subset\tau$, let $T_{\sigma,\tau}\subset\mathbb{R}^3$ be a tube that connects $B_\sigma$ and $B_\tau$. Since we are in $\mathbb{R}^3$, the tubes can always be arranged so that they do not intersect with each other and this can be proved by induction the number of tubes. Given $m$ disjoint tubes between $|\mathcal{C}\backslash\{\emptyset\}|$ balls, suppose we need to construct a tube $T_{\sigma,\tau}$ joining the balls $B_\sigma$ and $B_\tau$. Since the number $m$ of existing tubes is finite, we can pick a point $s \in B_\sigma$ and a point $t \in B_\tau$ such that their projections in the $(0,0,1)$ direction is larger than that of any other point on the $m$ existing tubes. Now join $s$ and $t$ by a tube $T_{\sigma,\tau}$ such that its projection onto the $(0,0,1)$ direction is larger than that of all other tubes. Thus, $T_{\sigma,\tau}$ is disjoint from the first $m$ tubes, completing the induction argument. Now, let
		\begin{equation*}A_\sigma:=B_\sigma \cup\bigcup_{\sigma\subset\tau}T_{\sigma,\tau}\end{equation*} for $\sigma\in\mathcal{C}\backslash\{\emptyset\}$ and let
		\begin{equation*}
		A_\emptyset:=\mathbb{R}^3\backslash\bigcup_{\sigma\in\mathcal{C}\backslash\{\emptyset\}}A_\sigma.
		\end{equation*}Then $\mathcal{C} = \mathcal{C}(\mathcal{A})$. Define $\mathcal{U}$ from $\mathcal{A}$ via (\ref{eqn:u.from.a}). By Lemma \ref{lem:u.from.a}, $\mathcal{C} = \mathcal{C}(\mathcal{U})$. By construction of $\mathcal{A}$ and since we are assuming that $\mathcal{C}$ satisfies Proposition \ref{prop:main}, the $U_i$'s are connected. This completes the proof.
	\end{proof}
	
	\begin{figure}[h]
		\begin{center}
			\includegraphics[width=13cm]{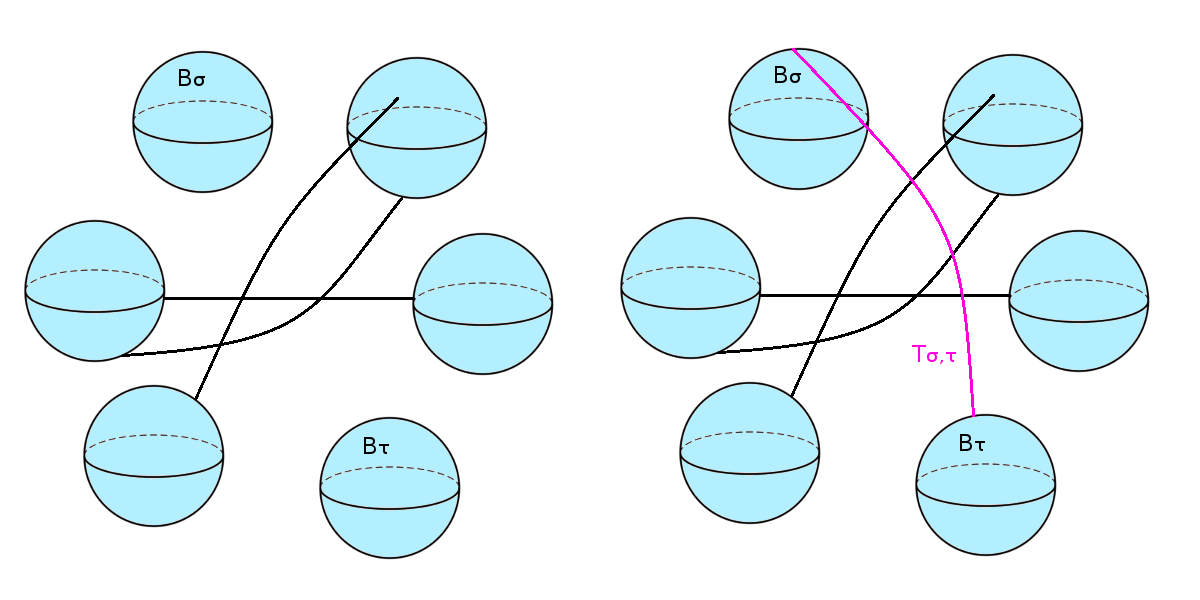}
		\end{center}
		\caption{Illustrative figure for the construction in Proposition \ref{prop:d3}. Given $m$ disjoint tubes between $|\mathcal{C}\backslash\{\emptyset\}|$ balls (picture on the left hand side), construct $T_{\sigma,\tau}$ such that its projection onto the $(0,0,1)$ direction is larger than that of all other tubes (right hand side).}
	\end{figure}
	
	\begin{proposition}\label{prop:d2}
		Let $\mathcal{C}$ be a connected code on $n$ neurons. Then $d^\ast(\mathcal{C}) = 2$ if and only if there exists a realization $\mathcal{C} = \mathcal{C}(\mathcal{U})$ by connected sets in $\R^3$ such that $\G(\mathcal{U})$ is planar.
	\end{proposition}
	
	\begin{proof}
		Suppose $d^\ast(\mathcal{C}) = 2$. Then there exists a realization $\mathcal{C} = \mathcal{C}(\mathcal{U})$ with $\mathcal{U}$ a collection of connected sets $\mathcal{U}$ in $\R^2$. The graph of $\mathcal{U}$, $\G(\mathcal{U})$, is by construction also embedded in $\R^2$. One can trivially embed a realization in $\R^2$ into $\R^3$ without changing the graph $\G(\mathcal{U})$, so we are done. Conversely, suppose that $\mathcal{C} = \mathcal{C}(\mathcal{U})$ for some $\mathcal{U}$ in $\R^3$ such that $\G(\mathcal{U})$ is planar. By Lemma \ref{lem:g.in.dprime}, $\mathcal{C}$ is realizable in dimension~$2$. 
	\end{proof} 
	
	We conclude our paper with two examples.
	
	\begin{example}[Connected code with $d^\ast = 3$]\label{ex:d3}
		Consider the following code 
		\begin{equation}\label{codedim2}
		\mathcal{C}=\{\emptyset,1,2,3,4,5,12,13,14,15,23,24,25,34,35,45\}.
		\end{equation} 
		This satisfies Proposition \ref{prop:main}, so $\mathcal{C}$ is a connected code. It's easy to see that every graph $\G(\mathcal{U})$ associated to this code must be a subdivision of the graph in Figure \ref{fig:ex.d3}, i.e. if $\mathcal{C}=\mathcal{C}(\mathcal{U})$, then $\G(\mathcal{U})$ must be either the graph in in Figure \ref{fig:ex.d3} or it can be obtained from it by subdividing some of its edges into two new edges, which must be connected to a new vertex. This is due to the fact that $\mathcal{C}$ is the code that contains exactly every $i\in[5]$ and every pair $ij$ of distinct $i,j\in[5]$. This implies, by Kuratowski's Theorem \cite{Kur}, that every graph associated to $\mathcal{C}$ is not planar. By Theorem \ref{thm:main}, $\mathcal{C}$ has minimal embedding dimension 3.
	\end{example}
	\begin{figure}[h]
		\begin{center}
			\includegraphics[width=5cm]{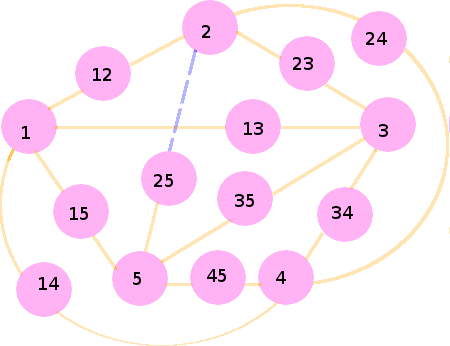}\end{center}
		\caption{The graph of a connected realization of a code $\mathcal{C}$ with $d^\ast(\mathcal{C}) = 3$ in Example \ref{ex:d3}.}\label{fig:ex.d3}
	\end{figure}
	
	\begin{example}[Connected code with $d^\ast = 2$]\label{ex:d2}
		Let $\mathcal{C}=\{\emptyset,1, 2, 3, 12, 123\}$ be a code on $3$ neurons. By Proposition \ref{prop:main}, this code is connected. Figure \ref{fig:d2} shows its realization by connected sets in $\R^2$, and the corresponding graph. We claim that the minimal embedding dimension of this code is $2$. One could verify by computing the sensor graph of $\mathcal{C}$. Alternatively, note that for the code to be realizable by connected sets, we must have $U_1\cap U_2 \cap U_3\neq\emptyset$ and $U_i$ can not be contained in $U_j$ for every $i,j\in[3]$, $i\neq j$. This is clearly not possible in dimension~$1$.
	\end{example}
	\begin{figure}[h]
		\begin{center}
			\includegraphics[width=12cm]{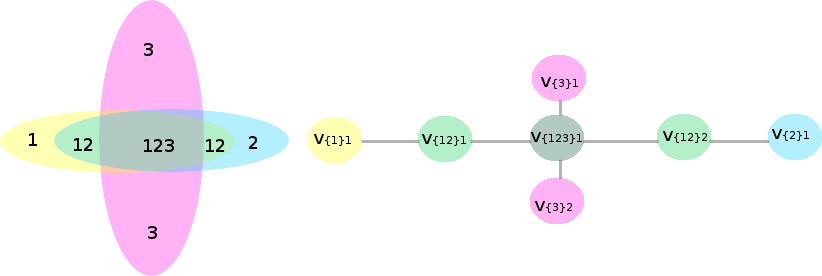}\end{center}
		\caption{The realization of a connected code $\mathcal{C}$ with $d^\ast(\mathcal{C}) = 2$ in Example \ref{ex:d2} and its graph.}\label{fig:d2}
	\end{figure}

\end{document}